\documentclass[reqno]{amsart}

\usepackage{amsbsy,amssymb,amscd,amsfonts,latexsym,amstext,delarray,
amsmath,mathrsfs,epsfig,graphicx}
\newenvironment{nitemize}
{\begin{list}{$\bullet$}
{\setlength{\parsep}{1ex}
\setlength{\topsep}{1.1ex}
\setlength{\partopsep}{0ex}
\setlength{\labelwidth}{0.5cm}
\setlength{\itemindent}{0cm}
\setlength{\itemsep}{0cm}
\setlength{\leftmargin}{0.7cm}
\setlength{\labelsep}{0.2cm}}}
{\end{list}}

\newenvironment{proof-theo}{\noindent{\it Proof of Theorem \ref{Kifer-implies-approx}.}}{\fbox{}\\}

\newenvironment{proof-corollary}{\noindent{\it Proof of Corollary \ref{criterium-dense-vector-space-unique-equilibrium}. }}{\fbox{}\\}

\newtheorem{theorem}{Theorem}
\newtheorem{proposition}{Proposition}

\newtheorem{lemma}{Lemma}

\theoremstyle{remark}

\theoremstyle{definition}

\theoremstyle{definition}

\def\N{{\mathbb N}}

\def\R{{\mathbb R}}

\def\Z{\mathbb Z}

\newcommand{\sS}{\mathscr{S}}

\newcommand{\ie}{{\it i.e.\/}\ }
\newcommand{\eg}{{\it e.g.\/}\ }
\newcommand{\cf}{{\it cf.\/}\ }

\def\text{\hbox}

\title[]{Entropy approximation  versus uniqueness of equilibrium  for  a dense affine  space of continuous functions}

\begin{document}

\author{Henri Comman$^\dag$}
\address{Pontificia Universidad Cat\'{o}lica de Valparaiso,  Avenida Brasil 2950, Valparaiso, Chile}
\email{henri.comman@pucv.cl}
\thanks{$\dag$ Partially supported by FONDECYT grant 1120493.}



\begin{abstract}
We show that for a $\mathbb{Z}^{l}$-action  (or $(\N\cup\{0\})^l$-action)   on a non-empty compact metrizable space $\Omega$,  the existence of a  affine  space dense in the set of continuous functions on $\Omega$ constituted by  elements admitting a unique equilibrium state implies that each invariant measure  can be approximated weakly$^*$ and in entropy by a sequence of   measures which are unique equilibrium states. 
  \end{abstract}


\maketitle

\section{Introduction}\label{introduction}\label{systems considered}
Let 
$\tau$  be an action of   $\mathbb{Z}^{l}$ (resp.  $(\N\cup\{0\})^l$) on  a non-empty  compact metrisable space  $\Omega$  for some $l\in\N$, let  $C(\Omega)$, $\mathcal{M}(\Omega)$,
$\mathcal{M}^{\tau}(\Omega)$, 
$h^\tau$, $P^\tau$ denote respectively the set of real-valued continuous functions on $\Omega$ endowed with the uniform topology, Borel  probability measures on $\Omega$ endowed with the weak-$^*$ topology, $\tau$-invariant elements of $\mathcal{M}(\Omega)$,  measure-theoretic entropy and pressure maps (\cite{Ruelle-78}). We assume that   $h^\tau$ is finite and  upper semi-continuous.

Let   $(\nu_\alpha,t_\alpha)$ be a net where $\nu_\alpha$ is  a 
  Borel probability measures on  $\mathcal{M}(\Omega)$, $t_\alpha>0$ and  $(t_\alpha)$ converges   to zero.    Recall that $(\nu_\alpha)$ is said to satisfy a large
deviation principle in  $\mathcal{M}(\Omega)$
with powers $(t_\alpha)$ if there exists a
$[0,+\infty]$-valued lower semi-continuous function $I$  on  $\mathcal{M}(\Omega)$
such
that
\[\limsup_\alpha t_\alpha\log\nu_{\alpha}(F)\le-\inf_{x\in F}I(x)\le-\inf_{x\in G}I(x)\le
 \liminf_\alpha t_\alpha\log\nu_\alpha(G)\]
for every closed set $F\subset\mathcal{M}(\Omega)$ and every open set $G\subset \mathcal{M}(\Omega)$ with $F\subset
G$; such a  function $I$ is unique  and called the rate function governing the large deviation principle.
Let $f\in C(\Omega)$ and assume that  $(\nu_\alpha,t_\alpha)$ fulfils 
     \begin{equation}\label{Introduction and statement of the result-eq20}
\forall g\in C(\Omega),\ \ \ \ \ \ \lim_\alpha t_\alpha\log\int_{\mathcal{M}(\Omega)}e^{t_\alpha^{-1}\int_{\Omega}g(\xi)\mu(d\xi)}\nu_\alpha(d\mu)=P^{\tau}(f+g)-P^{\tau}(f).
   \end{equation}
   There are two  general conditions  ensuring  that $(\nu_\alpha)$ satisfies a   large
deviation principle   with powers $(t_\alpha)$ and convex rate function: The first one is the existence 
of a vector space $V$  dense in $C(\Omega)$ 
  such that $f+g$ has a  unique equilibrium state for all $g\in V$ (\cite{Kifer-TAMS-90},  \cite{Comman_Rivera-Letelier(2010)ETDS31}); the second one, that we  shall  denote by (D) hereafter,  is the following entropy approximation property: For each $\mu\in\mathcal{M}^{\tau}(\Omega)$  there exists a net   $(\mu_i)$ in $\mathcal{M}^{\tau}(\Omega) $   such that
 $\lim_i\mu_i=\mu$, $\lim_i h^{\tau}(\mu_i)=h^{\tau}(\mu)$ and $\mu_i$ is the unique equilibrium state for some  $g_i\in C(\Omega)$
(\cite{Comman(2009)NON22}, Theorem 5.2); note that the  first countability of  $\mathcal{M}^{\tau}(\Omega)$   allows us to  replace "net" by "sequence" in (D);  since   a  measure on $\Omega$   is the unique equilibrium state for some element in $C(\Omega)$  if and only if it is ergodic  (\cite{Phelps_Dynamics and Randomness(2002)Santiago}),    one can also replace 
"unique equilibrium state for some  $g_i\in C(\Omega)$" by "ergodic"; in particular, (D) implies that either  $\mathcal{M}^{\tau}(\Omega)$ is a singleton or $\mathcal{M}^{\tau}(\Omega)$ is the Poulsen simplex, \ie the set of ergodic states is dense in  $\mathcal{M}^{\tau}(\Omega)$.

 In this note,  we prove that the first above condition implies the second one (Theorem \ref{Kifer-implies-approx}); consequently, all the large deviation principles proved applying Kifer's theorem  (namely, Theorem 2.1 of \cite{Kifer-TAMS-90}) or Theorem C of \cite{Comman_Rivera-Letelier(2010)ETDS31} (as well as its    generalization  given by  Remark  B.2)     can be proved  using   Theorem 5.2 of \cite{Comman(2009)NON22}.   We also provide two examples that can be proved using (D) but not with  the first condition (Example \ref{LDP-general-case-separated-sets}, Example \ref{LDP-general-case-periodic-points}).

The results  are given in Section \ref{results-examples}.   We  recall below some basic definitions and discuss    the main difference  between both  conditions; we also  review some  important cases where (D) has been used to prove a  large deviation principle.

\section{Preliminaries}\label{preliminaries}
\subsection{Pressure and equilibrium state}\label{Pressure and equilibrium state}
Let $(\Omega,\tau)$ be  a dynamical system as in \S \ref{systems considered}.  Order $\N^l$  lexicographically. For each   $a\in \N^l$ we put   $\Lambda(a)=\{(x_1,...,x_l)\in(\N\cup\{0\})^l:x_i<a_i,1\le i\le l\}$ and let ${|\Lambda(a)|}$ denote the cardinality of $\Lambda(a)$. For each $\varepsilon>0$ and for each
 $a\in \N^l$ let $\Omega_{\varepsilon,a}$ be
 a maximal
$(\varepsilon,\Lambda(a))$-separated set.
Recall that  $P^\tau(g)$ is defined for each $g\in C(\Omega)$ by
 \[
 P^{\tau}(g)=
\lim_{\varepsilon\rightarrow 0}\limsup_a\frac{1}{|\Lambda(a)|}\log\sum_{\xi\in
\Omega_{\varepsilon,a}}e^{\sum_{x\in\Lambda(a)}g(\tau^x\xi)},
\]
and fulfils
\begin{equation}\label{subsection-Thermodynamic formalism-eq20}
P^{\tau}(g)=
\lim_{\varepsilon\rightarrow 0}\liminf_a\frac{1}{|\Lambda(a)|}\log\sum_{\xi\in
\Omega_{\varepsilon,a}}e^{\sum_{x\in\Lambda(a)}g(\tau^x\xi)}=\sup_{\mu\in\mathcal{M}^{\tau}(\Omega)}\{\mu(g)+h^\tau(\mu)\}.
\end{equation}
(\cite{Ruelle-78}, \S 6.7, \S 6.12 and Exercise 2 p. 119 for $\Z^l$-action, \S 6.18 for $(\N\cup\{0\})^l$-action). 
Since $\mathcal{M}^{\tau}(\Omega)$ is compact and  $h^\tau$  is finite and upper semi-continuous, the above supremum is a maximum, and  each element  realizing this  maximum is called  an   equilibrium state for $g$; the map  $P^\tau$ is finite convex and continuous on $C(\Omega) $(\cite{Ruelle-78}, \S 6.8 and  \S 6.18).

\subsection{Connection with large deviation theory}\label{Connection with large deviation theory}
Each one  of the two  conditions stated in \S \ref{systems considered}  is nothing but the specialization in the dynamical setting of the corresponding well-known sufficient condition in large deviation theory in topological vector spaces that imply  the large deviation principle  (\cite{Dembo-Zeitouni}, \cite{Comman(2007)STAPRO77});  they appear virtually  in one form of another when  establishing a level-2 large deviation principle (\ie in the space $\mathcal{M}(\Omega)$). Regarding the first one,  the reader is referred  to \cite{Comman_Rivera-Letelier(2010)ETDS31},   where  Theorem C together with Remark B.2 
provides a version of   Theorem 2.1 of \cite{Kifer-TAMS-90} valid for general dynamical systems as in \S \ref{systems considered}; see also  \S 1.3 and Remark 3.3 of \cite{Comman_Rivera-Letelier(2010)ETDS31} for a discussion
 on the functional approach in large deviation theory in connection with dynamical systems, and specially with Corollary 4.6.14 of \cite{Dembo-Zeitouni}. 
 With respect to (D) and particularly the  connection with the general result involving exposed points in  large deviation theory in topological vector spaces (namely, Baldi's theorem),  see  \cite{Comman(2009)NON22},  where  Theorem 5.2 establishes    that (D)   yields  the large deviation principle  for any net $(\nu_\alpha,t_\alpha)$
   fulfilling (\ref{Introduction and statement of the result-eq20}).

 \subsection{Advantage of (D) over the first condition}\label{Advantage of (D) over the first condition}
 A straightforward but  important observation that differentiates the two above conditions is that   the first one    implies the uniqueness of equilibrium for  $f$, whereas (D) does not impose any condition on $f$ (\cf  "Important Remark" in \S 1.3 of  \cite{Comman(2009)NON22});
      this simple fact allows us to obtain  for a dynamical system $(\Omega,\tau)$ as in \S \ref{systems considered},   examples of large deviation principles that can be proved using (D) but that cannot be proved using the first condition: We just have to consider nets fulfilling (\ref{Introduction and statement of the result-eq20}) with $f$ admitting several equilibrium states;
      Example 4.1 of  \cite{Comman(2009)NON22} furnishes such an example  when  $(\Omega,\tau)$ is given by the iteration of a hyperbolic rational map;  Theorem 5.7 of  \cite{Comman(2009)NON22} provides other two   examples  when $(\Omega,\tau)$  is the multidimensional full shift.
      
     $\bullet$ Example 4.1  and Theorem 5.7 (a) of  \cite{Comman(2009)NON22}  both concern nets $(\nu_{f,\alpha},t_\alpha)$
       of the form 
      \begin{equation}\label{LDP-separated-sets-eq20}
\nu_{f,\alpha}=\sum_{\xi\in
\Omega_\alpha}
\frac{e^{\sum_{x\in\Lambda_\alpha} f(\tau^x\xi)}}
{\sum_{\xi'\in
\Omega_\alpha}  e^{\sum_{x\in\Lambda_\alpha}f(\tau^x\xi')}}
\delta_{\frac{1}{\mid\Lambda_\alpha\mid}\sum_{x\in
 \Lambda_\alpha}\delta_{\tau^x(\xi)}}.
 \end{equation}
 and 
 \[t_\alpha=\mid\Lambda_\alpha\mid ^{-1},\] 
 where $f$ is an arbitrary element of $C(\Omega)$, $(\Lambda_\alpha)$   a van Hove net of nonempty finite subsets of $\N^l$ for some $l\in\N$,  $\mid\Lambda_\alpha\mid$  the cardinality of $\Lambda_\alpha$,  and 
 $\Omega_\alpha$  a maximal  $(\varepsilon,\Lambda_\alpha)$-separated set for some $\varepsilon$ small enough;
 the  expansiveness property that holds in these examples makes   (\ref{Introduction and statement of the result-eq20})   easy to establish;   we present here the general case   (Example \ref{LDP-general-case-separated-sets}). 
 
 $\bullet$ The example given in Theorem 5.7(b) of  \cite{Comman(2009)NON22} deals with nets $(\nu_{f,\alpha},t_\alpha)$ similar to the above case but where $\alpha\in\N^l$, $\Lambda_\alpha$ is the parallelepiped whose angles are determined by $\alpha$, 
  and $\Omega_\alpha$ is the set of $\alpha$-periodic configurations; 
    Example \ref{LDP-general-case-periodic-points} 
    extends this case   to subshifts of finite type satisfying strong specification, 
  recovering 
    Theorem C of  \cite{Eiz-Kif-Weis-94-CMP} in a very direct way.

\subsection{Basic examples}\label{basis-examples}
   It should be pointed out that in all the examples below,  in addition to (D),  the proofs appearing  in the cited articles
  require   highly technical and or theoretical (\eg Shannon-Mac-Millan theorem) intermediate results,  while 
     once (D) is established,  the  large deviation principle follows by  a  straightforward application of Theorem 5.2 of  \cite{Comman(2009)NON22}; 
     this is in particular the case of Theorem A and Theorem C of \cite{Eiz-Kif-Weis-94-CMP}, which  follow from  Theorem B of \cite{Eiz-Kif-Weis-94-CMP}  together with Theorem 5.2 of  \cite{Comman(2009)NON22} (\cf b) and c)  below, respectively); the last case is detailed  in 
        Example \ref{LDP-general-case-periodic-points}.

\begin{nitemize}
\item[a)] In  statistical mechanics, the use of (D)   dates back to the proof of the large deviation principle for Gibbs random  fields on $\Z^l$ for some $l\in\N$ (\cite{Follmer-88-AnnProb}), where the underlying dynamical system $(\Omega,\tau)$ is the $l$-dimensional full shift and the function $f$ as in (\ref{Introduction and statement of the result-eq20}) is  the local energy function  associated with some translation invariant  summable interaction $\phi$; more precisely, given a  van Hove sequence   $(\Lambda_n)$  of finite subsets of $\Z^l$,
one considers the sequence  $(\nu_n,t_n)$, where  $t_n=\mid\Lambda_n\mid^{-1}$ and  $\nu_n$ is
the  distribution of the   field
\[\Omega\ni\xi\mapsto\frac{1}{\mid\Lambda_n\mid}\sum_{x\in\Lambda_n}\delta_{\tau^x\xi}\]
induced by an equilibrium state  $\mu_\phi$ for $f_\phi$ \ie
\begin{equation}\label{Birkhoff-average-meca-stat}
\forall n\in\N,\ \ \ \ \ \ \ \nu_n(\cdot)=\mu_{f_\phi}(\{\xi\in\Omega:\frac{1}{\mid\Lambda_n\mid}\sum_{x\in\Lambda_n}
\delta_{\tau^x \xi}\in\cdot\}). 
\end{equation}
The fact that (D) holds  for  full shifts  is known for a long time (\cite{Israel-79}, Lemma IV,  3.2); refinements have been given in \cite{Gurevich-Tempelman-05-PTRF}  showing that the approximating measures may be chosen as full-supported  equilibrium measures for local energy functions associated with invariant short-range interactions.

\item[b)] The   foregoing   is a particular case of a general result for   $\Z^l$-actions with upper semi-continuous entropy:  indeed,  Theorem B of  \cite{Eiz-Kif-Weis-94-CMP} asserts that if such a system satisfies the  weak specification property,  then (D) holds;
in \cite{Eiz-Kif-Weis-94-CMP} the authors apply this result to   prove the large deviation principle for the same sequence $(\nu_n,t_n)$ as in  $a)$ above,  but  in the more general case  where $(\Omega,\tau)$ is a 
 subshift of finite type satisfying weak specification  (\cite{Eiz-Kif-Weis-94-CMP}, Theorem A); it turns out that for these  sequences,  
 the weak specification also implies  (\ref{Introduction and statement of the result-eq20}) (\ie  equation (2.23) of \cite{Eiz-Kif-Weis-94-CMP}). In fact, \cite{Eiz-Kif-Weis-94-CMP} deals with the more  general  van Hove net constituted by all finite subsets of $\Z^l$ ordered by inclusion, 
  case which  can be   reduced   to the one   of sequences (thanks to the weak specification).

 \item[c)] In \cite{Eiz-Kif-Weis-94-CMP} the large deviation principle is also proved for subshifts of finite type satisfying strong specification
 and for  the  net $(\nu_a,t_a)_{a\in\N^l}$ given  for each $a\in\N^l$ by
  \[\nu_a=\frac{1}{\mid\textnormal{Per}_a\mid}\sum_{\xi\in\textnormal{Per}_a}\delta_{\frac{1}{\mid\Lambda(a)\mid}\sum_{x\in
 \Lambda(a)}\delta_{\tau^x(\xi)}}\]
 and
 \[t_a=\mid\Lambda(a)\mid^{-1},\]
where   $\textnormal{Per}_a$ denotes the set of $a$-periodic configurations (\cite{Eiz-Kif-Weis-94-CMP}, Theorem C). 
The possible several equilibrium states for $f=0$ as an obstacle to the application of the first condition (\ie results of \cite{Kifer-TAMS-90}) has been noted by the authors who, instead,  apply the  large deviation principle obtained previously in the above case $b)$ (namely, Theorem A of \cite{Eiz-Kif-Weis-94-CMP}).

\item[d)] In one dimensional  dynamics, let us consider the  system  $(\Omega,\tau)$ constituted  by the iteration of a rational map $T$ of degree at least two (\cite{Beardon-91});  more precisely, 
$\Omega$ is the Julia set of $T$ endowed with the induced chordal metric, and  the action   $\tau$ is defined by
\[\N\cup\{0\}\ni n\mapsto\tau(n) = (T_{\mid\Omega})^{n};\]
such a system  has a unique measure of maximal entropy (\cite{Ljubich-83-ETDS-3}); 
we assume furthermore that $T$ fulfils a weak  form of hyperbolicity,  the so-called  ÒTopological Collet-EckmanÓ (TCE) condition: There exists $\lambda> 1$ such that every periodic point $p\in\Omega$ with period $n$ satisfies
\begin{equation}\label{Introduction and statement of the result-eq40}
\mid (T^n)'(p)\mid\ge\lambda^n.
\end{equation}
(see Main Theorem of \cite{Przytycki_Rivera-Letelier_Smirnov(2003)InventMath151}
 for other equivalent definitions).  It is known that (D)  holds when $T$ is  hyperbolic (\ie  when (\ref{Introduction and statement of the result-eq40}) holds for all $p\in\Omega$) (\cite{lopes-90-NON3}, Theorem 8); 
 in \cite{lopes-90-NON3} the author uses (D) to prove the large deviation principle for the  sequence  $(\nu_n,n^{-1})$, where $\nu_n$ is   the  distribution of the  Birkhoff averages with respect to the  measure of maximal entropy $\mu_0$, \ie
\[\forall n\in\N,\ \ \ \ \ \ \ 
\nu_n (\cdot)=\mu_0
\{\xi\in\Omega:\frac{1}{n}\sum_{k=0}^{n-1}
\delta_{\tau^k \xi}\in\cdot\}.
\]

\item[e)]  The above example $d)$ has been  our  starting point in  \cite{Comman(2009)NON22} leading to the general Theorem 5.2. 
In the case of the multidimensional full shift, the  example of Theorem 5.7(b) above mentioned  in  \S \ref{Advantage of (D) over the first condition} 
generalizes also 
          Theorem 3.5 and  Theorem 4.2 of \cite{Olla-88-PTRF}  in the  case of  a finite spin space with uniform  single spin measure, 
by allowing $f$ arbitrary in $C(\Omega)$ (and not only $f=0$) (\cf Remark 5.8 of \cite{Comman(2009)NON22} for more details).
 Other  examples  are given in \S 4 of \cite{Comman(2009)NON22} (and in particular  those considered in \cite{Pollicott-Sridharan-07-JDSGT} and \cite{pol_CMP_96}, where only the large deviation upper bounds are proved) but  where $f$ is assumed  to have a unique equilibrium state and thus for which the large deviation principle can be proved as well 
 using the first condition (namely, Theorem C together with Remark B.2 of \cite{Comman_Rivera-Letelier(2010)ETDS31});  see also Remarks 5.4 and 5.5 of \cite{Comman(2009)NON22}.
 \end{nitemize}

\section{Results}\label{results-examples}

Our main result is the following.

\begin{theorem}\label{Kifer-implies-approx}
   Let $(\Omega,\tau)$ be  a dynamical system as in \S \ref{systems considered}. Let $f\in C(\Omega)$. We assume that   there exists a 
   vector space $V$ dense in $C(\Omega)$ such that $f+g$ has a  unique equilibrium state for all $g\in V$. 
      Then,   for each $\mu\in\mathcal{M}^{\tau}(\Omega)$  there exists a sequence    $(\mu_n,g_n)$ in $\mathcal{M}^{\tau}(\Omega)\times V$   such that
 $\lim_n\mu_n=\mu$, $\lim_n h^{\tau}(\mu_n)=h^{\tau}(\mu)$ and $\mu_n$ is the equilibrium state for $f+g_n$ for all $n\in\N$.
 \end{theorem}

The proof of Theorem \ref{Kifer-implies-approx} will be given after establishing a few lemmas.

\begin{lemma}\label{lemma-conv-net}
Let $J$ and $L$ be  directed sets,  let $s$ be a real-valued function on $J\times L$, let  $\wp$ denote the set $J\times L^{J}$ pointwise directed,  let 
  $(s_i)_{i\in\wp}$ be the  net in $\R$ defined by putting $s_i=s(j,u(j))$ for all  $i=(j,u)\in\wp$. For each  $r\in\R$ we have
\[\limsup_i\ s_i\le
r\Longleftrightarrow\limsup_j\limsup_l\
s({j,l})\le r;\]
in particular,
\[\lim_i\ s_i=r\Longleftrightarrow\liminf_j\liminf_l\
s({j,l})=\limsup_j\limsup_l\
s({j,l})= r.\]
\end{lemma}

\begin{proof}
Let $\delta>0$. First assume that $\limsup_i\ s_i\le r$. There
exists $j_0\in J$ and $u_0\in L^J$
such that $s({j,u(j)})<r+\delta$  for all
$(j,u)$ greater than or equal  $(j_0,u_0)$. Suppose that
$\limsup_j\limsup_l\ s({j,l})>r+\delta$. There exists   $(j_1,l_1)$ in $J\times L$ with $j_1$ (resp. $l_1$)   greater than or equal $j_0$ (resp. $u_0(j_1)$)   such that     $s({j_1, l_1})>r+\delta$. Putting
  $u_1(j_1)=l_1$ and $u_1(j)=u_0(j)$   for all $j\in J\setminus\{l_1\}$,  we get  an element $(j_1,u_1)\in\wp$ greater than or equal $(j_0,u_0)$ fulfilling
$s({j_1,u_1(j_1)})>r+\delta$, which gives  the
contradiction. Therefore, we have  $\limsup_j\limsup_l\
s({j,l})\le r+\delta$ hence $\limsup_j\limsup_l\
s({j,l})\le r$  since $\delta$ is
arbitrary. 

Assume now that  $\limsup_j\limsup_l\
s({j,l})\le r$.  There exists $j_0\in J$  and for each $j\in J$ greater than or equal $j_0$ there exists 
  $u_0(j)\in L$ such that  
$s({j,l})<r+\delta$ for all $j$ and $l$  greater than or equal $j_0$ and $u_0(j)$, respectively.
Putting  $u_0(j)=u_0(j_0)$  for all  $j$ lesser than  $j_0$, we get an element  $(j_0,u_0)\in\wp$ such that 
$s({j,u(j)})<r+\delta$ for all $(j,u)\in\wp$ greater than or equal $(j_0,u_0)$; therefore,  $\limsup_i s_i\le r+\delta$ hence 
$\limsup_i s_i\le r$  since $\delta$ is
arbitrary. The first assertion is proved; the second assertion is a direct consequence since 
 $\liminf_i\ s_i\ge r$ if and only if $-\limsup_i\ -s_i\ge r$  if and only if $\limsup_i\ -s_i\le -r$ if and only if $\limsup_j\limsup_l\
-s({j,l})\le -r$ if and only if $-\liminf_j\liminf_l\
s({j,l})\le -r$ if and only if $\liminf_j\liminf_l\
s({j,l})\ge r$ (where the third equivalence follows from the  first assertion applied to the net $(-s_i)$ and the real  $-r$).
\end{proof}

Let $f\in C(\Omega)$. Let $Q$ be the map defined on $C(\Omega)$ by 
\[\forall g\in C(\Omega),\ \ \ \ \ \ \ \ \ Q(g)=P^{\tau}(f+g)-P^{\tau}(f).\]

\begin{lemma}\label{Fenchel-Legendre-transform-of-Q}
The function $Q$ is proper convex and continuous;  its Fenchel-Legendre transform $Q^*$  has effective domain 
 $\mathcal{M}^{\tau}(\Omega)$ and fulfils 
\[\forall\mu\in\mathcal{M}^{\tau}(\Omega),\ \ \ \ \ \ \ \  Q^*(\mu)=P^{\tau}(f)-h^{\tau}(\mu)-\mu(f).\]
In particular,  $Q^*$ vanishes exactly on the set of equilibrium states for $f$. 
\end{lemma}

\begin{proof}
Clearly, $Q$ is proper  convex and continuous  since $P^{\tau}$ and $\widehat{f}$ are (\cf \S \ref{Pressure and equilibrium state}).  Let $\widetilde{\mathcal{M}}(\Omega)$ denote the space of signed Radon measures on $\Omega$ endowed with the weak-$^*$ topology. 
 Putting
\[
U(\mu)=\left\{
\begin{array}{ll}
-\mu(f)-h^{\tau}(\mu) & \textnormal{if $\mu\in\mathcal{M}^{\tau}(\Omega)$}
\\ 
+\infty & \textnormal{if $\mu\in\widetilde{\mathcal{M}}(\Omega)\setminus \mathcal{M}^{\tau}(\Omega)$},
\end{array}
\right.
\]
 we have
\[P^{\tau}(f+g)=\sup_{\mathcal{M}^{\tau}(\Omega)}\{\mu(f+g)+h^{\tau}(\mu)\}=
\sup_{\mu\in\widetilde{\mathcal{M}}(\Omega)}
\{\mu(g)-U(\omega)\}.\]
Since $h^{\tau}$ is bounded  affine and upper semi-continuous,  $U$ is proper convex and  lower semi-continuous; consequently, we have $U=U^{**}$,  \ie 
\[\forall \mu\in\widetilde{\mathcal{M}}(\Omega),\ \ \ \ \ \ \ U(\mu)=\sup_{g\in C(\Omega)}\{\mu(g)-P^{\tau}(f+g)\}=
\sup_{g\in C(\Omega)}\{\mu(g)-P^{\tau}(f)-Q(g)\}=\]
\[-P^{\tau}(f)+\sup_{g\in C(\Omega)}\{\mu(g)-Q(g)\}=-P^{\tau}(f)+Q^*(\mu),\]which
proves the lemma. 
 \end{proof}

  For each
$d\in\N$, each $S=(g_1,...,g_d)\in C(\Omega)^d$ and each $t=(t_1,...,t_d)$
in $\mathbb{R}^d$, let  $tS$  denote the function
 $\sum_{i=1}^d t_i g_i$,  put
  $L_S(t)=P^{\tau}(f+tS)-P^{\tau}(f)$, let
$p_S:\mathcal{M}(\Omega)\rightarrow\mathbb{R}^d$ defined by
$p_S(\mu)=(\mu(g_1),...,\mu(g_d))$ for all $\mu\in\mathcal{M}(\Omega)$, and let
 $I_S:\mathbb{R}^d\rightarrow[0,+\infty]$ defined by
 \begin{displaymath}\label{ext-Kifer-eq10}
I_S(x)=\left\{
\begin{array}{ll}
\inf\{Q^*(\mu):\mu\in\mathcal{M}^{\tau}(\Omega),p_S(\mu)=x\} &
\textnormal{if $x\in p_S(\mathcal{M}^{\tau}(\Omega))$}
\\
+\infty & \textnormal{otherwise};
\end{array}
\right.
\end{displaymath}
note that since $\mathcal{M}^{\tau}(\Omega)$ is compact and $Q^*$  is lower semi-continuous and real-valued on $\mathcal{M}^{\tau}(\Omega)$ (\cf Lemma \ref{Fenchel-Legendre-transform-of-Q}), for each  $x\in p_S(\mathcal{M}^{\tau}(\Omega)$ there exists   $\mu_x\in\mathcal{M}^{\tau}(\Omega)$ such that  
$I_S(x)=Q^*(\mu_x)$. Let $(d,S)\in\N\times C(\Omega)^d$.

\begin{lemma}\label{IS-convex-lsc}
The function 
 $I_S$ is proper convex and lower semi-continuous. 
 \end{lemma}

\begin{proof}
Let $x\in\R^d$,  let $(x_i)$ be a net converging to $x$, and
assume $\liminf I_S(x_i)<\delta$ for some real $\delta$. For some
subnet $(x_j)$ we have eventually $I_S(x_j)<\delta$ and so
$Q^*(\mu_j)<\delta$ for some $\mu_j\in\mathcal{M}^{\tau}(\Omega)$
satisfying $p_S(\mu_j)=x_j$. Let $(\mu'_j)$ be a subnet of $(\mu_j)$
converging to some $\mu'\in\mathcal{M}^{\tau}(\Omega)$; note
that $p_S(\mu')=x$. We have
\[I_S(x)\le Q^*(\mu')\le\liminf\  Q^*(\mu'_j)<\delta,\]
which proves the lower semi-continuity of $I_S$. For each $(x_1,x_2,\beta)\in\R^{2d}\times\ ]0,1[$
  we have
\[I_S(\beta x_1+(1-\beta)x_2)=\inf\{Q^*(\mu):
\mu\in\mathcal{M}^{\tau}(\Omega),p_S(\mu)=\beta
x_1+(1-\beta)x_2\}\]
\[\le\inf\{Q^*(\beta\mu_1+(1-\beta)\mu_2):\mu_1\in\mathcal{M}^{\tau}(\Omega),
\mu_2\in\mathcal{M}^{\tau}(\Omega),p_S(\mu_1)=x_1,p_S(\mu_2)=x_2\}\]
\[\le\inf\{\beta Q^*(\mu_1)+(1-\beta)Q^*(\mu_2):\mu_1\in\mathcal{M}^{\tau}(\Omega),
\mu_2\in\mathcal{M}^{\tau}(\Omega),p_S(\mu_1)=x_1,p_S(\mu_2)=x_2\}\]
\[=\beta I_S(x_1)+(1-\beta)I_S(x_2),\] hence $I_S$ is convex; $I_S$ is proper by the observation following its definition.
\end{proof}

\begin{lemma}\label{IS*=LS}
$I_S=L^*_S$.
 \end{lemma}

\begin{proof}
 Let $\langle\ ,\ \rangle$ denote the scalar product in $\R^d$.
Suppose $\sup_{x\in\mathbb{R}^d}\{\langle
t,x\rangle-I_S(x)\}<L_S(t)$ for some $t\in\mathbb{R}^d$. Since
\[L_S(t)=Q(tS)=Q^{**}(tS)=\sup_{\mu\in\mathcal{M}^{\tau}(\Omega)}\{\langle t,p_S(\mu)\rangle -Q^*(\mu)\}\]
 there exists $\mu\in\mathcal{M}^{\tau}(\Omega)$ such that
$\sup_{x\in\mathbb{R}}\{\langle t,x\rangle -I_S(x)\}<\langle
t,p_S(\mu)\rangle-Q^*(\mu)$, which gives the contradiction by taking
$x=p_S(\mu)$ in the left hand side. Conversely, if
$\sup_{x\in\mathbb{R}}\{\langle t,x\rangle -I_S(x)\}>L_S(t)$ for
some  $t\in\mathbb{R}^d$, then $\langle
t,x\rangle-I_S(x)>\sup_{\mu\in\mathcal{M}^{\tau}(\Omega)}\{\langle
t,p_S(\mu)\rangle-Q^*(\mu)\}$ for some $x=p_S(\mu)$ with
$\mu\in\mathcal{M}^{\tau}(\Omega)$, which gives the contradiction. Therefore, we have 
$I^*_S=L_S$, which is equivalent to  the conclusion since  $I_S$  is convex proper and 
lower semi-continuous by Lemma \ref{IS-convex-lsc}.
\end{proof}

Let 
  $\textnormal{dom\ }\delta L_S^*$ denotes the set of
$x\in\mathbb{R}^d$ for which the set $\delta L_S^*(x)$  of
subgradients of $L_S^*$ at $x$ is nonempty.

\begin{lemma}\label{subgradient-Q*-vs-subgradient-IS}
For each $(x,t)\in\textnormal{dom\ }\delta L_S^*\times\delta L_S^*(x)$ we have $x=p_S(\mu)$ where $\mu$ is an 
 equilibrium state for  $f+tS$; moreover,
$L_S^*(p_S(\mu))=Q^*(\mu)$.
\end{lemma}

\begin{proof}
Let $(x,t)\in\textnormal{dom\ }\delta L_S^*\times\delta L_S^*(x)$. 
Necessarily $x$ belongs to the effective domain of $L_S^*$, and so by Lemma \ref{IS*=LS} 
 there exists
$\mu\in\mathcal{M}^{\tau}(\Omega)$ such that $x=p_S(\mu)$ and
$L_S^*(x)=Q^*(\mu)$. Consequently, we have
\[\langle t,x\rangle-L_S^*(x)=L_S(t)=Q(tS)=\langle t,p_S(\mu)\rangle -Q^*(\mu)=
\mu(tS)-Q^*(\mu),\] which proves the lemma.
\end{proof}

\begin{proof-theo}
Let $V$ be a vector space as  in  Theorem \ref{Kifer-implies-approx}.   Since $C(\Omega)$ is second countable,  $V$ contains a countable  set $\{g_n:n\in\N\}$ dense in $C(\Omega)$; put $W=\textnormal{span}(\{g_n:n\in\N\})$. 
Let $\mu\in\mathcal{M}^{\tau}(\Omega)$. 
For each $n\in\N$
  we  put  $S_n=(g_1,...,g_n)$, $x_n=p_{S_n}(\mu)$;   note that 
$x_n$ belongs to the effective domain of $L_{S_n}^*$ by Lemma \ref{IS*=LS}. 
For each
 $n\in\N$,  Br\"{o}ndsted-Rockafellar theorem (\cite{Brondsted_Rockafellar(1965)TAMS16}, Theorem 2) ensures the existence of 
 a sequence $(x_{n,m})$ in $\textnormal{dom\ }\delta L_{S_n}^*$ such
that $\lim_m x_{n,m}=x_n$ and  $\lim_m L^*_{S_n}(x_{n,m})= L^*_{S_n}(x_n)$; by
Lemma \ref{subgradient-Q*-vs-subgradient-IS},  we have $x_{n,m}=p_{S_n}(\mu_{n,m})$ and
$L^*_{S_n}(x_{n,m})=Q^*(\mu_{n,m})$, where $\mu_{n,m}$ is the unique 
equilibrium state for some $f+t_{n,m} S_n$, with $t_{n,m}\in\delta
L_{S_n}^*(x_{n,m})$; therefore, we have
\begin{equation}\label{ext-Kifer-eq5}
\forall (n,g)\in \N\times\textnormal{span}(\{g_1,...,g_n\}),\ \ \ \ \ \ \lim_m\mu_{n,m}(g)=\mu(g)
\end{equation}
 and
\begin{equation}\label{ext-Kifer-eq10}
\forall n\in\N,\ \ \ \ \ \ \ \ \ Q^*(\mu)\ge I_{S_n}(x_n)=L^*_{S_n}(x_n)=\lim_m Q^*(\mu_{n,m}).
\end{equation}
Since each $g\in W$ belongs  to $\textnormal{span}(\{g_1,...,g_n\})$ for all $n$ large enough, (\ref{ext-Kifer-eq5})   yields 
 \begin{equation}\label{ext-Kifer-eq20}
 \forall g\in W,\ \ \ \ \ \ \ \ \lim_n\lim_m\mu_{n,m}(g)=\mu(g).
 \end{equation}
By considering the product  set
$\wp=\N\times\mathbb{N}^\N$ pointwise directed, we obtain  a net
$(\mu_i)_{i\in\wp}$ in $\mathcal{M}^{\tau}(\Omega)$ defined
for each $i=(n,u)\in\wp$ by  $\mu_i=\mu_{n,u(n)}$; 
since $W$ is dense in $C(\Omega)$, (\ref{ext-Kifer-eq20}) yields
\begin{equation}\label{ext-Kifer-eq25}
\lim_i\mu_i=\mu
\end{equation}
(\cite{Kelley-91}, Theorem on Iterated Limits, p. 69).
Let  $s$ be the function  defined on $\N^2$ by 
\[\forall (n,m)\in\N^2,\ \ \ \ \ \ \ s(n,m)=Q^*(\mu_{n,m});\]
note that $s$ is real-valued by Lemma \ref{Fenchel-Legendre-transform-of-Q}.
The lower semi-continuity of $Q^*$ and  (\ref{ext-Kifer-eq25}) yield
 \begin{equation}\label{ext-Kifer-eq30}
\liminf_i Q^*(\mu_i)\ge Q^*(\mu).
\end{equation}
  Since     $Q^*(\mu)\ge\limsup_n\lim_m Q^*(\mu_{n,m})$ by  (\ref{ext-Kifer-eq10}),  we have
       \begin{equation}\label{ext-Kifer-eq40}
Q^*(\mu)\ge\limsup_i Q^*(\mu_i)
\end{equation}
by  Lemma
\ref{lemma-conv-net}.
From   (\ref{ext-Kifer-eq30}) and  (\ref{ext-Kifer-eq40}) we get 
$\lim_i Q^*(\mu_i)=Q^*(\mu)$, \ie $\lim_i h^\tau(\mu_i)=h^\tau(\mu)$ by Lemma \ref{Fenchel-Legendre-transform-of-Q}, 
which together with (\ref{ext-Kifer-eq25}) shows that the net  $(\mu_i,h^\tau(\mu_i))$ converges to $(\mu,h^\tau(\mu))$. 
Denoting  $\sS$  the subset of $\mathcal{M}^{\tau}(\Omega)$ constituted by  the measures that are unique equilibrium states for some element in $f+W$,  we have proved that the graph of ${h^\tau}_{\mid \sS}$  is dense in the graph of $h^\tau$; the conclusion follows since the graph of $h^\tau$ is a 
  first countable space. 
\end{proof-theo}

\subsection{Examples of large deviation principles as consequence of (D)}

In this section we present two examples where the large deviation principle  is  a direct consequence of (D), but where 
 the first condition may not be  fulfilled,  and thus that cannot be proved using  
 Theorem C and Remark B.2 of \cite{Comman_Rivera-Letelier(2010)ETDS31} neither with  Theorem 2.1 of \cite{Kifer-TAMS-90}. 
 They are both obtained from   a  net $(t_\alpha,m_\alpha)$ generating the pressure in the sense that 
 \[\forall g\in C(\Omega),\ \ \ \ \ \ \ \ \ 
\lim_\alpha t_\alpha\log\int_{\mathcal{M}(\Omega)}e^{t_\alpha^{-1}\int_{\Omega}g(\xi)\mu(d\xi)}m_\alpha(d\mu)=P^{\tau}(g),
   \]
   which yields  (\ref{Introduction and statement of the result-eq20}) after normalization and the choice of an arbitrary function $f\in C(\Omega)$ (so that no vector space $V$ can fulfil the first condition stated in \S \ref{systems considered}
  when $f$ admits several equilibrium state); however, assuming (D),   
    the large deviation principle  follows as a straightforward application of Theorem 5.2 of \cite{Comman(2009)NON22}.

\subsubsection{Property (D) and maximal separated sets}\label{LDP-general-case-separated-sets}
We generalize   Theorem 5.7(a) and Example 4.1  of  \cite{Comman(2009)NON22} to any  dynamical system  $(\Omega,\tau)$ 
as in \S \ref{systems considered}.  

  Let  $\wp$ denote    the product  set
$]0,+\infty[\times {\N^l}^{]0,+\infty[}$ pointwise directed, where $]0,+\infty[$ (resp. $\N$,  $\N^l$)  is endowed  with the inverse of the  natural order on $\R$ (resp. natural order, lexicographic order), \ie $(\varepsilon,u)\in\wp$ is less than or equal  $(\varepsilon',u')\in\wp$ if $\varepsilon\ge\varepsilon'$ and $u(\delta)$ is lexicographically  less than or equal $u'(\delta)$ for all $\delta\in\ ]0,+\infty[$ (\cf \cite{Kelley-91}).
For each  $\alpha=(\varepsilon,u)\in\wp$ we put 
$\Lambda_\alpha=\Lambda(u(\varepsilon))$,  $\Omega_\alpha=\Omega_{\varepsilon,u(\varepsilon)}$ and 
\[\forall f\in C(\Omega),\ \ \ \ \ \ \ \ \nu^\tau_{f,\alpha}=\sum_{\xi\in
\Omega_\alpha}\frac{e^{\sum_{x\in\Lambda_\alpha}f(\tau^x\xi)}}{\sum_{\xi'\in
\Omega_\alpha}e^{\sum_{x\in\Lambda_\alpha}f(\tau^x\xi')}}\delta_{\frac{1}{\mid\Lambda_\alpha\mid}\sum_{x\in
 \Lambda_\alpha}\delta_{\tau^x(\xi)}}.\]

 \begin{proposition}\label{general-example-LDP-maximal-separated-set}
 If (D) holds, then for each $f\in C(\Omega)$ the net $(\nu^\tau_{f,\alpha})$
 satisfies a   large
deviation principle in $\mathcal{M}(\Omega)$ with powers $(\mid\Lambda_\alpha\mid^{-1})$ and  rate function
\[
\mathcal{M}(\Omega)\ni\mu\mapsto
\left\{
\begin{array}{ll}
P^{\tau}(f)-\mu(f)-h^{\tau}(\mu) & \textnormal{if $\mu\in\mathcal{M}^{\tau}(\Omega)$}
\\ 
+\infty & \textnormal{if $\mu\in\mathcal{M}(\Omega)\setminus \mathcal{M}^{\tau}(\Omega)$}.
\end{array}
\right.
\]
\end{proposition}

\begin{proof}
For each $g\in C(\Omega)$ let $s_g$ be the real-valued map on $]0,+\infty[\times\N^l$
defined by 
  \[\forall (\varepsilon,a)\in\  ]0,+\infty[\times\N^l,\ \ \ \ \ \ \ \ s_g(\varepsilon,a)=\frac{1}{|\Lambda (a)|}\log\sum_{\xi\in
\Omega_{\varepsilon,a}}e^{\sum_{x\in\Lambda(a)}g(\tau^x\xi)}.\]
Then,  (\ref{subsection-Thermodynamic formalism-eq20}) together with  Lemma \ref{lemma-conv-net} (applied with $J=\ ]0,+\infty[$,   $L=\N^l$  and the above function  $s_g$) yields
\[\forall g\in C(\Omega),\ \ \ \ \ \ \ \ P^{\tau}(g)=\lim_\alpha \frac{1}{|\Lambda_\alpha|}\log\sum_{\xi\in
\Omega_\alpha}e^{\sum_{x\in\Lambda_\alpha}g(\tau^x\xi)}\]
hence
\[\forall (f,g)\in C(\Omega)^2,\ \ \ \ \ \ \ 
\lim_\alpha \frac{1}{|\Lambda_\alpha|}\log\int_{\mathcal{M}(\Omega)}
e^{(t^\tau_\alpha)^{-1}\int_\Omega g(\omega)\mu(d\omega)}\nu^\tau_{f,\alpha}(d\mu)=P^{\tau}(f+g)-P^{\tau}(f)\]
 (\ie (\ref{Introduction and statement of the result-eq20}) holds with $(\nu_\alpha,t_\alpha)=(\nu^\tau_{f,\alpha},{|\Lambda_\alpha|}^{-1}$); the conclusion follows from Theorem 5.2 of \cite{Comman(2009)NON22}.
     \end{proof}

\subsubsection{Subshifts of finite type  and periodic points}\label{LDP-general-case-periodic-points}

The following example  illustrates how direct is the use of (D)  in order to get a   large deviation principle in comparison with usual proofs: indeed, the conclusion  of Proposition \ref{subshift-finite-type-LDP-periodic-points} when $f=0$  is exactly  Theorem C of  \cite{Eiz-Kif-Weis-94-CMP} (note that the fact that we get at once the general  case  $f\in C(\Omega)$ is just a bonus since it    follows from the case $f=0$ by a standard  result in large deviation theory, \cf Appendix B in \cite{Ellis-1985}).

  Let   $(\Omega,\tau)$ be a $l$-dimensional  subshift of finite type satisfying strong specification (\cite{Ruelle-73-TAMS}, \cite{Eiz-Kif-Weis-94-CMP}). For each $a\in\N^l$  we put
\[\forall f\in C(\Omega),\ \ \ \ \ \ \ \ \nu_{f,a}=\sum_{\xi\in\textnormal{Per}_a}
\frac{e^{\sum_{x\in\Lambda(a)}f(\tau^x\xi)}}{\sum_{\xi'\in
\textnormal{Per}_a }e^{\sum_{x\in\Lambda(a)}f(\tau^x\xi')}}\delta_{\frac{1}{\mid\Lambda(a)\mid}\sum_{x\in
 \Lambda(a)}\delta_{\tau^x(\xi)}},\] 
 where  $\textnormal{Per}_a$ denote the set of $a$-periodic points.

\begin{proposition}\label{subshift-finite-type-LDP-periodic-points}
For each $f\in C(\Omega)$
the net $(\nu_{f,a})$
 satisfies a   large
deviation principle in $\mathcal{M}(\Omega)$ with powers $(\mid\Lambda(a)\mid^{-1})$ and  rate function
\[
\mathcal{M}(\Omega)\ni\mu\mapsto
\left\{
\begin{array}{ll}
P^{\tau}(f)-\mu(f)-h^{\tau}(\mu) & \textnormal{if $\mu\in\mathcal{M}^{\tau}(\Omega)$}
\\ 
+\infty & \textnormal{if $\mu\in\mathcal{M}(\Omega)\setminus \mathcal{M}^{\tau}(\Omega)$},
\end{array}
\right.
\]
\end{proposition}

\begin{proof}
 Theorem 2.2 of \cite{Ruelle-73-TAMS}   yields
\[\forall g\in C(\Omega),\ \ \ \ \ \ \ \ \ \ P^{\tau}(g)=\lim_a \frac{1}{|\Lambda(a)|}\log\sum_{\xi\in\textnormal{Per}_a}
e^{\sum_{x\in\Lambda(a)}g(\tau^x\xi)}
\]
hence 
\[\forall (f,g)\in C(\Omega)^2,\ \ \ \ \ \ \ \ \ \ \ 
\lim_a \frac{1}{|\Lambda(a)|}\log\int_{\mathcal{M}(\Omega)}
e^{(t_a)^{-1}\int_\Omega g(\omega)\mu(d\omega)}\nu_{f,a}(d\mu)=P^{\tau}(f+g)-P^{\tau}(f).\]
Since the  strong specification  implies the weak specification,    and since $h^\tau$ is upper semi-continuous by expansiveness, (D) holds by   Theorem B of  \cite{Eiz-Kif-Weis-94-CMP}; the conclusion follows from Theorem 5.2 of \cite{Comman(2009)NON22}.
\end{proof}

\end{document}